\documentclass[12pt]{article}
\usepackage{amsmath,amssymb,amsfonts,amsthm}

\usepackage{graphicx}
\newenvironment{myabstract}{\par\noindent
{\bf Abstract . } \small }
{\par\vskip8pt minus3pt\rm}
\newcounter{item}[section]
\newcounter{kirshr}
\newcounter{kirsha}
\newcounter{kirshb}
\newenvironment{enumroman}{\setcounter{kirshr}{1}
\begin{list}{(\roman{kirshr})}{\usecounter{kirshr}} }{\end{list}}
\newenvironment{enumarab}{\setcounter{kirshb}{1}
\begin{list}{(\arabic{kirshb})}{\usecounter{kirshb}} }{\end{list}}
\newenvironment{athm}[1]{\vskip3mm\par\noindent
{\bf #1 }. \slshape }
{\upshape\par\vskip10pt minus3pt}
\newtheorem{theorem}{Theorem}[section]

\newtheorem{lemma}[theorem]{Lemma}
\newtheorem{corollary}[theorem]{Corollary}
\newtheorem{question}[theorem]{Question}

\newenvironment{demo}[1]{\noindent{\bf #1.}\upshape\mdseries}
{\nopagebreak{\hfill\rule{2mm}{2mm}\nopagebreak}\par\normalfont}
\theoremstyle{definition}

\newtheorem{example}[theorem]{Example}
\newtheorem{definition}[theorem]{Definition}

\def\Nr{{\mathfrak{Nr}}}
\def\Fr{{\mathfrak{Fr}}}
\def\Sg{{\mathfrak{Sg}}}
\def\Fm{{\mathfrak{Fm}}}

\def\CA{{\bf CA}}

\def\RCA{{\bf RCA}}

\def\(R)RA{{\bf (R)RA}}

\def\C{\mathbb{C}}
\def\A{{\mathfrak{A}}}
\def\B{{\mathfrak{B}}}
\def\C{{\mathfrak{C}}}

\def\PA{{\bf PS}}

\def\Gp{{\sf Gp}}

\def\PA{{\bf PA}}

\def\CPEA{{\sf CPEA}}
\def\L{{\mathfrak{L}}}

\def\Nr{{\mathfrak{Nr}}}
\def\Fr{{\mathfrak{Fr}}}
\def\Sg{{\mathfrak{Sg}}}
\def\Fm{{\mathfrak{Fm}}}

\def\CA{{\bf CA}}

\def\RCA{{\bf RCA}}

\def\(R)RA{{\bf (R)RA}}

\def\C{\mathbb{C}}
\def\A{{\mathfrak{A}}}
\def\B{{\mathfrak{B}}}
\def\C{{\mathfrak{C}}}

\def\M{{\mathfrak{M}}}

\def\PA{{\bf PA}}

\def\P{{\mathfrak P}}

\def\At{{\sf At}}
\def\PA{{\sf PA}}
\def\RPA{{\sf RPA}}
\def\Cm{{\sf Cm}}
\title{An atomic polyadic algebra of infinite dimension is completely representable if and only if is completely additive} 
\author{Tarek Sayed Ahmed\\
Department of Mathematics, Faculty of Science,\\ 
Cairo University, Giza, Egypt.
  }
\begin{document}
\maketitle

\begin{myabstract}  Answering a question posed by Hodkinson, we show that for infinite ordinals $\alpha$, every atomic polyadic algebra of 
dimension $\alpha$ $(\PA_{\alpha})$
is completely representable if and only if it is completely additive. 
We readily infer that the class of completely representable $\PA_{\alpha}$s is elementary.
This is in sharp contrast to the cylindric case.
However, we do not settle the question as to whether there are atomic polyadic algebras that are not completely additive, hence not completely
representable.
In this connection, we could only show that the  proper reduct of polyadic algebras, obtained by discarding all non bijective transformations, 
with the exception of replacements, is not completely additive.
Our proof of the equivalence in the title uses 
a neat embedding theorem together with a simple topological argument, namely, that principal ultrafilters in the Stone space of a 
Boolean algebra lie outside nowhere dense sets, and if the algebra is atomic they form a dense subset of the Stone topology. 
An analogous result is proved for many modifications of polyadic algebras. 
In all cases the signature contains all substitutions, so that the cylindrifier free reduct
of such algebras, can be viewed as a transformation system.
Finally, we give a metalogical
interpretation to our algebraic result, which is a Vaught's theorem for Keisler's logic.  Certain atomic theories have atomic models. 
Based on work of Ferenczi, we discuss various metalogical property for reducts of Keisler's logic endowed with equality. 
In particular, for such logic, we show that any atomic theory has an atomic model.
Further generalizations are discussed.
\footnote{ 

keywords: Algebraic Logic, polyadic algebras, complete representations}
 
\end{myabstract}

Polyadic algebras were introduced by Halmos \cite{Halmos} to provide an algebraic reflection
of the study of first order logic without equality. Later the algebras were enriched by 
diagonal elements to permit the discussion of equality. 
That the notion is indeed an adequate reflection of first order logic was 
demonstrated by Halmos' representation theorem for locally finite polyadic algebras 
(with and without equality). Daigneault and Monk  
proved a strong extension of Halmos' theorem, namely that,  
every polyadic algebra of infinite dimension (without equality) is representable \cite{DM}. 

There are several types of representations in algebraic logic. Ordinary representations are just isomorphisms from boolean algebras
with operators to a more concrete structure (having the same similarity type) 
whose elements are sets endowed with set-theoretic operations like intersection and complementation.
Complete representations, on the other hand, are representations that preserve arbitrary conjunctions whenever defined.
The notion of complete representations has turned out to be very interesting for cylindric algebras, where it is proved in \cite{HH}
that the class of completely representable algebras is not elementary. 

Lately, it has become fashionable to study representations  that preserve infinitary meets and joins. 
This phenomena is extensively discussed in \cite{Sayed}, where it is shown that it
has affinity with the algebraic notion of complete representations  for cylindric like algebras
and atom-canonicity in varieties of Boolean algebras with operators, 
a prominent persistence property studied in modal logic. 

The typical question is: given an  algebra  and a set of meets, is there a representation  that carries 
this set of meets to set theoretic intersections? 
(assuming that our semantics is specified by set algebras, with the concrete Boolean operation of intersection among
its basic operations.) 
When the algebra in question is countable, and we have countably many meets; 
this is an algebraic version of an  omitting types theorem; the representation omits the given set meets or non-principal types.  
When it is only one meet consisting of co-atoms, in an atomic algebra, this representation is a complete one.

The correlation of atomicity to complete representations has caused a lot of confusion in the past.
It was mistakenly thought for a long time, among algebraic logicians,  that atomic representable relation and cylindric algebras
are completely representable, an error attributed to Lyndon and now referred to as Lyndon's error.

For boolean algebras, however  this is true.  The class of completely representable algebras
is simply the class of atomic ones. An analogous result holds for certain countable reducts of polyadic algebras \cite{s}.
The notion of complete representations has been linked to Martin's axiom, omitting types theorems 
and existence of atomic models in model theory \cite{t}, \cite{ANS}, \cite{Moh4}. Such connections w
ill be worked out below in the context of Keisler's logic
the infinitary logic corresponding to $\PA_{\alpha}$. This logic allows formulas of infinite length and infinite quantification.

In this paper we show that an atomic polyadic algebra of infinite dimension is also completely representable, if and only if it is completely additive.
From this we conclude that the class of completely representable algebras is elementary by just spelling out first order formulas, 
one for each substitution stipulating that it is completely additive. This gives continuum many formulas, however, they
share one schema.

Our result is in sharp and, indeed, interesting contrast to the cylindric and quasi-polyadic equality cases
\cite{HH}, \cite{Moh1} (these are completely additive varieties). This result also adds to the long list of results
existing in the literature, further emphasizing, the commonly accepted view 
that cylindric algebras and polyadic algebras belong to different paradigms often exhibiting 
conflicting behaviour.
It also  answers a question raised by Ian Hodkinson, 
see p. 260 in \cite{h} and Remark 6.4, p. 283 in op.cit, and a question in \cite{1}, though admittedly the latter is 
ours \cite{Sayed}.

\section{ Main result}

Our notation is in conformity with \cite{1} which is based on the notation \cite{HMT2}. Hoever, we chose to deviate
from \cite{HMT2}, when we felt that it was compelling to reverse the order. 
We write $f\upharpoonright A$ instead of $A\upharpoonright f$, 
for the restriction of a function $f$ to 
a set $A,$ which is the more usual standard notation. 
On the other hand, following \cite{HMT2}, for given sets $A, B$ we write $A\sim B$ for the set $\{x\in A: x\notin B\}$.
Gothic letters are used for algebras, and the corresponding Roman letter will denote their universes.
Also for an algebra $\A$ and $X\subseteq A$, $\Sg^{\A}X$, or simply $\Sg X$ when $\A$ is clear from context, denotes the subalgebra of $\A$ 
generated by $X$. $Id$ denotes the identity function and when we write $Id$ we will be tacitly assuming that its domain is clear from context.
We now recall the definition of polyadic algebras as formulated in \cite{HMT2}. 
Unlike Halmos' formulation, the dimension of algebras is specified by ordinals as opposed to 
arbitrary sets.

\begin{definition} Let $\alpha$ be an ordinal. By a {\it polyadic algebra} of dimension $\alpha$,  
or a $\PA_{\alpha}$ for short,
we understand an algebra of the form
$$\A=\langle A,+,\cdot ,-,0,1,{\sf c}_{(\Gamma)},{\sf s}_{\tau} \rangle_{\Gamma\subseteq \alpha ,\tau\in {}^{\alpha}\alpha}$$
where ${\sf c}_{(\Gamma)}$ ($\Gamma\subseteq  \alpha$) and ${\sf s}_{\tau}$ ($\tau\in {}^{\alpha}\alpha)$ are unary 
operations on $A$, such that postulates  
below hold for $x,y\in A$, $\tau,\sigma\in {}^{\alpha}\alpha$ and 
$\Gamma, \Delta\subseteq \alpha$ 

\begin{enumerate}
\item $\langle A,+,\cdot, -, 0, 1\rangle$ is a boolean algebra
\item ${\sf c}_{{(\Gamma)}}0=0$

\item $x\leq {\sf c}_{{(\Gamma)}}x$

\item ${\sf c}_{(\Gamma)}(x\cdot {\sf c}_{(\Gamma)}y)={\sf c}_{(\Gamma)}x \cdot {\sf c}_{(\Gamma)}y$

\item ${\sf c}_{(\Gamma)}{\sf c}_{(\Delta)}x={\sf c}_{(\Gamma\cup \Delta)}x$

\item ${\sf s}_{\tau}$ is a boolean endomorphism

\item ${\sf s}_{Id}x=x$

\item ${\sf s}_{\sigma\circ \tau}={\sf s}_{\sigma}\circ {\sf s}_{\tau}$

\item if $\sigma\upharpoonright (\alpha\sim \Gamma)=\tau\upharpoonright (\alpha\sim \Gamma)$, then 
${\sf s}_{\sigma}{\sf c}_{(\Gamma)}x={\sf s}_{\tau}{\sf c}_{(\Gamma)}x$


\item If $\tau^{-1}\Gamma=\Delta$ and $\tau\upharpoonright \Delta $ is
one to one, then ${\sf c}_{(\Gamma)}{\sf s}_{\tau}x={\sf s}_{\tau}{\sf c}_{(\Delta )}x$.

\end{enumerate}

\end{definition}
We will sometimes add superscripts to cylindrifiers and substitutions indicating the algebra they are evaluated in.
The class of representable algebras is defined via set - theoretic operations
on sets of $\alpha$-ary sequences. Let $U$ be a set. 
For $\Gamma\subseteq \alpha$ and $\tau\in {}^{\alpha}\alpha$, we set
$${\sf c}_{(\Gamma)}X=\{s\in {}^{\alpha}U: \exists t\in X,\ \forall j\notin \Gamma, t(j)=s(j) \}$$
and 
$${\sf s}_{\tau}X=\{s\in {}^{\alpha}U: s\circ \tau\in X\}.$$
For a set $X$, let $\B(X)$ be the boolean set algebra $(\wp(X), \cup, \cap, \sim).$
The class of representable polyadic algebras, or ${\RPA}_{\alpha}$ for short, is defined by
$$SP\{\langle \B(^{\alpha}U), {\sf c}_{(\Gamma)}, {\sf s}_{\tau} \rangle_{\Gamma\subseteq \alpha, \tau\in {}^{\alpha}\alpha}:
\ \,  U \text { a set }\}.$$
Here $SP$ denotes the operation of forming subdirect products. It is straightforward to show that ${\RPA}_{\alpha}\subseteq \PA_{\alpha}.$ 
Daigneault and Monk \cite{DM} proved that for $\alpha\geq \omega$ the converse inclusion also holds, 
that is ${\RPA}_{\alpha}=\PA_{\alpha}.$ 
This is a completeness theorem for certain infinitary extensions of first order logic 
without equality \cite{K}.

In this paper we are concerned with the following question:
If $\A$ is a polyadic algebra,  is there a 
representation of $\A$ that preserves infinite meets and joins, whenever they exist? 
(A representation of a given abstract algebra is basically a non-trival 
homomorphism from this algebra into a set algebra).
To make the problem more tangible we need to prepare some more.
In what follows $\prod$ and $\sum$ denote infimum
and supremum, respectively. We will encounter situations where we need to evaluate a supremum of a 
given set in more than one algebra, in which case we will add a superscript to the supremum indicating the algebra we want. 
For set algebras, we identify notationally the algebra with its universe, since the operations 
are uniquely defined given the unit of the algebra.

Let $\A$ be a polyadic algebra and $f:\A\to \wp(^{\alpha}U)$ be a representation of $\A$.
If $s\in X$, we let
$$f^{-1}(s)=\{a\in \A: s\in f(a)\}.$$
An atomic representation $f:\A\to \wp(^{\alpha}U)$ is a representation such that for each 
$s\in V$, the ultrafilter $f^{-1}(s)$ is principal. 
A complete representation of $\A$ is a representation $f$ satisfying
$$f(\prod X)=\bigcap f[X]$$
whenever $X\subseteq \A$ and $\prod X$ is defined.

A completely additive boolean algebra with operators is one for which all extra non-boolean operations preserve arbitrary joins.

\begin{lemma}\label{r} Let $\A\in \PA_{\alpha}$. A representation $f$ of $\A$
is atomic if and only if it is complete. If $\A$ has a complete representation, then it is atomic and is completely additive.
\end{lemma}
\begin{demo}{Proof} The first part is like \cite{HH}.
For the second part, we note that $\PA_{\alpha}$ is a discriminator variety with discriminator term ${\sf c}_{(\alpha)}$.
And so because all algebras in $\PA_{\alpha}$ are semi-simple,
it suffices to show that if $\A$ is simple, $X\subseteq A$, is such that $\sum X=1$,
and there exists  an injection $f:\A\to \wp(^{\alpha}\alpha)$, such that $\bigcup_{x\in X}f(x)=V$,
then for any $\tau\in {}^{\alpha}\alpha$, we have $\sum s_{\tau}X=1$. So assume that this does not happen 
for some $\tau \in {}^{\alpha}\alpha$.
Then there is a $y\in \A$, $y<1$, and
$s_{\tau}x\leq y$ for all $x\in X$.
Now
$$1=s_{\tau}(\bigcup_{x\in X} f(x))=\bigcup_{x\in X} s_{\tau}f(x)=\bigcup_{x\in X} f(s_{\tau}x).$$
(Here we are using that $s_{\tau}$ distributes over union.)
Let $z\in X$, then $s_{\tau}z\leq y<1$, and so $f(s_{\tau}z)\leq f(y)<1$, since $f$ is injective, it cannot be the case that $f(y)=1$.
Hence, we have
$$1=\bigcup_{x\in X} f(s_{\tau}x)\leq f(y) <1$$
which is a contradiction, and we are done.

\end{demo}
By Lemma \ref{r} a necessary condition for the existence of complete representations is the condition of atomicity and complete
additivity.
We now prove a converse to this result, namely, that when $\A$ is atomic and completely additive, then 
$\A$ is completely representable. 

We need to recall from \cite[definition~5.4.16]{HMT2},  the notion of neat reducts of polyadic algebras, 
which will play a key role in our proof of the main theorem. 
\begin{definition} Let $J\subseteq \beta$ and 
$\A=\langle A,+,\cdot ,-, 0, 1,{\sf c}_{(\Gamma)}, {\sf s}_{\tau}\rangle_{\Gamma\subseteq \beta ,\tau\in {}^{\beta}\beta}$
be a $\PA_{\beta}$.
Let $Nr_J\B=\{a\in A: {\sf c}_{(\beta\sim J)}a=a\}$. Then
$$\Nr_J\B=\langle Nr_{J}\B, +, \cdot, -, {\sf c}_{(\Gamma)}, {\sf s}'_{\tau}\rangle_{\Gamma\subseteq J, \tau\in {}^{\alpha}\alpha}$$
where ${\sf s}'_{\tau}={\sf s}_{\bar{\tau}}.$ Here $\bar{\tau}=\tau\cup Id_{\beta\sim \alpha}$.
The structure $\Nr_J\B$ is an algebra, called the {\it $J$ compression} of $\B$.
When $J=\alpha$, $\alpha$ an ordinal, then $\Nr_{\alpha}\B\in \PA_{\alpha}$ and is called the {\it neat $\alpha$ reduct} of $\B$ 
and its elements are called 
$\alpha$-dimensional.
\end{definition}

The notion of neat reducts is also extensively studied for cylindric algebras \cite{neat}.
We also need, \cite [theorem~ 2.1]{DM} and top of  p.161 in op.cit:

\begin{definition} Let $\A\in \PA_{\alpha}$. 
\begin{enumroman}

\item If $J\subseteq \alpha$, an element $a\in A$ is {\it independent of $J$} if ${\sf c}_{(J)}a=a$;
$J$ supports $a$ if $a$ is independent of $\alpha\sim J$.

\item The {\it effective degree} of $\A$ is the smallest cardinal $\mathfrak{e}$ such that each element of $\A$ admits a support 
whose cardinality does not exceed $\mathfrak{e}$.
\item The {\it local degree} of $\A$ is the smallest cardinal $\mathfrak{m}$ such that each element of $\A$ has support of cardinality $<\mathfrak{m}$. 
\item The {\it effective cardinality} of $\A$ is $\mathfrak{c}=|Nr_JA|$ where $|J|=\mathfrak{e}.$ (This is independent of $J$). 
\end{enumroman}
\end{definition} 

We chose to highlight he following simple basic known facts about boolean algebras and topological spaces. 
\begin{enumarab}
\item Let $\B$ be a boolean algebra, 
and let $S$ be its Stone space whose underlying set consists of all ultrafilters of $\B$. The topological space $S$ has a clopen base 
of sets of the form $N_b=\{F\in S: b\in F\}$ for $b\in B$. Assume that $X\subseteq B$ and $c\in B$
are such that $\sum X=c$. Then the set $N_c\sim \bigcup_{x\in X}N_x$ is nowhere dense in the 
Stone topology. In particular, if $c$ is the top element, then it follows that $S\sim \bigcup_{x\in X}N_x$ is nowhere dense.
(A nowhere dense set is one whose closure has empty interior).
\item Let ${\bf X}=(X, \tau)$ be a topological space. Let $x\in X$ be an isolated point in the sense that there is an open set $G\in \tau$ containing $x$, 
such that
$G\cap X=\{x\}$. Then $x$ cannot belong to any nowhere dense subset of $X$.  
\end{enumarab}
The proofs of these very elementary facts are entirely straightforward. They follow from the basic definitions.
Now we formulate and prove the main result. The proof is basically a Henkin construction together with a simple 
topological argument. The proof also has affinity with the proofs of the main theorems in \cite{DM} and \cite{super}, 
endowed with a topological 
twist. 

The idea is simple. Start with an atomic completely additive $\A\in \PA_{\alpha}$. Then $\A$ neatly embeds into an algebra
$\B$ having  enough spare dimensions, called a dilation of $\A$, that is $\A=\Nr_{\alpha}\B$. 
As it turns out, $\B$ is also atomic, and by complete additivity the sum of all all substituted versions of the set of atoms is the top element in $\B$. 
The desired representation is built from any principal ultrafilter thet preserves this set of infinitary joins 
as well as some  infinitary joins that have to do with eliminating cylindrifiers. A principal 
ultrafilter preserving  these sets of  joins 
can always  be found because, on the one hand, the set of principal ultrafilters are dense in the Stone space of the Boolean reduct of 
$\B$ since the latter is atomic,  and on the other hand, finding an ultrafilter preserving the given set  joints amounts to finding a 
a principal ultrafilter outside a nowhere dense set corresponding to the infinitary joins. Any such ultrafilter can be used to build the desired representation.
But first a definition: 

\begin{definition} A transformation system is a quadruple of the form $(\A, I, G, {\sf S})$ where $\A$ is an 
algebra of any similarity type, 
$I$ is a non empty set (we will only be concerned with infinite sets),
$G$ is a subsemigroup of $(^II,\circ)$ (the operation $\circ$ denotes composition of maps) 
and ${\sf S}$ is a homomorphism from $G$ to the semigroup of endomorphisms of $\A$. 
Elements of $G$ are called transformations. 
\end{definition}
The cylindrfier free reducts of 
polyadic algebras can be viewed as transformation systems where $I$ is the dimension and $G={}^II.$
Now we prove our main result:

\begin{theorem}\label{main} Let $\alpha$ be an infinite ordinal. Let $\A\in \PA_{\alpha}$ be atomic and completely additive.
Then $\A$ has a complete representation.
\end{theorem}

\begin{demo}{Proof} Let $c\in A$ be non-zero. We will find a set $U$ and a homomorphism 
from $\A$ into the set algebra with universe $\wp(^{\alpha}U)$ that preserves arbitrary suprema
whenever they exist and also satisfies that  $f(c)\neq 0$. $U$ is called the base of the set algebra. 
Let $\mathfrak{m}$ be the local degree of $\A$, $\mathfrak{c}$ its effective cardinality and 
$\mathfrak{n}$ be any cardinal such that $\mathfrak{n}\geq \mathfrak{c}$ 
and $\sum_{s<\mathfrak{m}}\mathfrak{n}^s=\mathfrak{n}$. The cardinal 
$\mathfrak{n}$ will be the base of our desired representation.

Substitutions in $\A$, induce a homomorphism 
of semigroups $S:{}^\alpha\alpha\to End(\A)$, via $\tau\mapsto {\sf s}_{\tau}$.  
The operation on both semigroups is composition of maps; the latter is the semigroup of endomorphisms on 
$\A$. For any set $X$, let $F(^{\alpha}X,\A)$ 
be the set of all functions from $^{\alpha}X$ to $\A$ endowed with boolean  operations defined pointwise and for 
$\tau\in {}^\alpha\alpha$ and $f\in F(^{\alpha}X, \A)$, put ${\sf s}_{\tau}f(x)=f(x\circ \tau)$. 
This turns $F(^{\alpha}X,\A)$ to a transformation system as well. 
The map $H:\A\to F(^{\alpha}\alpha, \A)$ defined by $H(p)(x)={\sf s}_xp$ is
easily checked to be an isomorphism. Assume that $\beta\supseteq \alpha$. Then $K:F(^{\alpha}\alpha, \A)\to F(^{\beta}\alpha, \A)$ 
defined by $K(f)x=f(x\upharpoonright \alpha)$ is an isomorphism. These facts are straighforward to establish, cf. theorem 3.1, 3.2 
in \cite{DM}. 
Call $F(^{\beta}\alpha, \A)$ a minimal functional dilation of $F(^{\alpha}\alpha, \A)$. 
Elements of the big algebra, or the (cylindrifier free) 
functional dilation, are of form ${\sf s}_{\sigma}p$,
$p\in F(^{\beta}\alpha, \A)$ where $\sigma$ is one to one on $\alpha$, cf. \cite{DM} theorem 4.3-4.4.

We can assume that $|\alpha|<\mathfrak{n}$.
Let $\B$ be the algebra obtained from $\A$, by discarding its cylindrifiers, then taking  a minimal functional dilation 
and then re-defining cylindrifiers in the bigger algebra, by setting for each $\Gamma\subseteq \mathfrak{n}:$
$${\sf c}_{(\Gamma)}{\sf s}_{\sigma}^{\B}p={\sf s}_{\rho^{-1}}^{\B} {\sf c}_{\rho(\{(\Gamma)\}\cap \sigma \alpha)}^{\A}{\sf s}_{(\rho\sigma\upharpoonright \alpha)}^{\A}p.$$
Here $\rho$ is a any permutation such that $\rho\circ \sigma(\alpha)\subseteq \sigma(\alpha.)$
The definition is sound, that is, it is independent of $\rho, \sigma, p$; furthermore, it agrees with the old cylindrifiers in $\A$.
Identifying algebras with their transformation systems
we have $\A\cong \Nr_{\alpha}\B$, via $H$ defined 
for $f\in \A$ and $x\in {}^{\beta}\alpha$ by, 
$H(f)x=f(y)$ where $y\in {}^{\alpha}\alpha$ and $x\upharpoonright \alpha=y$, 
cf. \cite{DM} theorem 3.10.

The local degree of $\B$ is the same as that of $\A$, 
in particular, each $x\in \B$ admits a support of cardinality $<\mathfrak{m}$. Furthermore, $|\mathfrak{n}\sim \alpha|=|\mathfrak{n}|$ and
for all $Y\subseteq A$, we have $\Sg^{\A}Y=\Nr_{\alpha}\Sg^{\B}Y.$ 
All this can be found in \cite{DM}, see the proof of theorem 1.6.1 therein; in such a proof, 
$\B$ is called a minimal dilation of $\A$, due to the fact that $\B$ is unique up to isomorphisms that fix $\A$ pointwise. 
 
We show that, like $\A$, $F=F(^{\mathfrak{m}}\alpha, \A)$, hence the boolean reduct of $\B$, is atomic.
Let $a$ be a non-zero element in $F$. Then $a:{}^{\mathfrak{m}}\alpha\to \A$, such that $a(t)\neq 0$ for some $t\in {}^{\mathfrak{m}}\alpha$.
Choose an atom $b(t)$ in $\A$ below $a(t)$ and define $b:{}^{\mathfrak{m}}\alpha\to \A$ 
via $t\mapsto b(t)$, and otherwise $b(t)=0$. 
Then clearly $b$ is an atom in $F$ below $a$, and so $\B$ is atomic.
(Note that the fact that $\A$ is the full neat reduct of $\B$ and that $\A$ generates $\B$ is not enough to show that atomicity 
of $\A$, implies that of $\B$ (see example \ref{ex} below). 
But in the context of polyadic algebras we are lucky; the dilations of atomic algebras are constructed in such a way 
that they are atomic, as well.) 

Let $\Gamma\subseteq \alpha$ and $p\in \A$. Then in $\B$ we have, see \cite{DM} the proof of theorem 1.6.1,  
\begin{equation}\label{tarek1}
\begin{split}
{\sf c}_{(\Gamma)}p=\sum\{{\sf s}_{\bar{\tau}}p: \tau\in {}^{\alpha}\mathfrak{n},\ \  \tau\upharpoonright \alpha\sim\Gamma=Id\}.
\end{split}
\end{equation}
Here, and elsewhere throughout the paper,  for a transformation $\tau$ wth domain $\alpha$ and range included in $\mathfrak{n}$,
$\bar{\tau}=\tau\cup Id_{\mathfrak{n}\sim \alpha}$. 
Let $X$ be the set of atoms of $\A$. Since $\A$ is atomic, then  $\sum^{\A} X=1$. By $\A=\Nr_{\alpha}\B$, we also have $\sum^{\B}X=1$.
By complete additivity we have 
for all $\tau\in {}^{\alpha}\mathfrak{n}$ we have,
\begin{equation}\label{tarek2}
\begin{split}
\sum {\sf s}_{\bar{\tau}}^{\B}X=1.
\end{split}
\end{equation}
Let $S$ be the Stone space of $\B$, whose underlying set consists of all boolean ulltrafilters of 
$\B$. Let $X^*$ be the set of principal ultrafilters of $\B$ (those generated by the atoms).  
These are isolated points in the Stone topology, and they form a dense set in the Stone topology since $\B$ is atomic. 
So we have $X^*\cap T=\emptyset$ for every nowhere dense set $T$ (since principal ultrafilters, which are isolated points in the Stone topology,
lie outside nowhere dense sets). 
For $a\in \B$, let $N_a$ denote the set of all boolean ultrafilters containing $a$.
Now  for all $\Gamma\subseteq \alpha$, $p\in B$ and $\tau\in {}^{\alpha}\mathfrak{n}$, we have,
by the suprema, evaluated in (1) and (2):
\begin{equation}\label{tarek3}
\begin{split}
G_{\Gamma,p}=N_{{\sf c}_{(\Gamma)}p}\sim \bigcup_{{\tau}\in {}^{\alpha}\mathfrak{n}} N_{s_{\bar{\tau}}p}
\end{split}
\end{equation}
and
\begin{equation}\label{tarek4}
\begin{split}
G_{X, \tau}=S\sim \bigcup_{x\in X}N_{s_{\bar{\tau}}x}.
\end{split}
\end{equation}
are nowhere dense. 
Let $F$ be a principal ultrafilter of $S$ containing $c$. 
This is possible since $\B$ is atomic, so there is an atom $x$ below $c$; just take the 
ultrafilter generated by $x$.
Then $F\in X^*$, so $F\notin G_{\Gamma, p}$, $F\notin G_{X,\tau},$ 
for every $\Gamma\subseteq \alpha$, $p\in A$
and $\tau\in {}^{\alpha}\mathfrak{n}$.
Now define for $a\in A$
$$f(a)=\{\tau\in {}^{\alpha}\mathfrak{n}: {\sf s}_{\bar{\tau}}^{\B}a\in F\}.$$
Then $f$ is a homomorphism from $\A$ to the full
set algebra with unit $^{\alpha}\mathfrak{n}$, such that $f(c)\neq 0$. We have $f(c)\neq 0$ because $c\in F,$ so $Id\in f(c)$. 
The rest can be proved exactly as in \cite{super}; the preservation of the boolean operations and substitutions is fairly 
straightforward. Preservation of cylindrifications
is guaranteed by the condition that $F\notin G_{\Gamma,p}$ for all $\Gamma\subseteq \alpha$ and all $p\in A$. (Basically an elimination 
of cylindrifications, this condition is also used in \cite{DM}
to prove the main representation result for polyadic algebras.)
Moreover $f$ is an 
atomic representation since $F\notin G_{X,\tau}$ for every $\tau\in {}^{\alpha}\mathfrak{n}$, 
which means that for every $\tau\in {}^{\alpha}\mathfrak{n}$, 
there exists $x\in X$, such that
${\sf s}_{\bar{\tau}}^{\B}x\in F$, and so $\bigcup_{x\in X}f(x)={}^{\alpha}\mathfrak{n}.$ 
We conclude that $f$ is a complete  representation by Lemma \ref{r}.
\end{demo}

Let ${\sf CPA}_{\alpha}$ denote the class of polyadic algebras of dimension $\alpha$.
Contrary to cylindric algebras, we have:

\begin{corollary} The class ${\sf CPA}_{\alpha}$ is elementary, and it is axiomatizable by a finite schema in first order logic.
\end{corollary}
\begin{demo}{Proof} Atomicity can be expressed by a first order sentence, and complete additivity can 
be captured by the following continuum many formulas,
that form a single schema.
Let $\At(x)$ be the first order formula expressing that $x$ is an atom. That is $\At(x)$ is the formula
$x\neq 0\land (\forall y)(y\leq x\to y=0\lor y=x)$. For $\tau\in {}^{\alpha}\alpha$, let $\psi_{\tau}$ be the formula:
$$y\neq 0\to \exists x(\At(x)\land {\sf s}_{\tau}x\neq 0\land {\sf s}_{\tau}x\leq y).$$ 
Let $\Sigma$ be the set of first order formulas obtained  by adding all formulas $\psi_{\tau}$ $(\tau\in {}^{\alpha}\alpha)$
to the polyadic schema.
We show that ${\sf CPA}_{\alpha}={\bf Mod}(\Sigma)$. Let $\A\in {\sf CPA}_{\alpha}$. Then, by theorem \ref{r}, we have
$\sum_{x\in X} {\sf s}_{\tau}x=1$ for all $\tau\in {}^{\alpha}\alpha$. Let $\tau \in {}^{\alpha}\alpha$.
Let $a$ be non-zero, then $a\cdot \sum_{x \in X}{\sf s}_{\tau}x=a\neq 0$,
hence there exists $x\in X$, such that
$a\cdot {\sf s}_{\tau}x\neq 0$, and  so $\A\models \psi_{\tau}$.
Conversely, let $\A\models \Sigma$. Then for all $\tau\in {}^{\alpha}\alpha$, $\sum_{x\in X} {\sf s}_{\tau}x=1$. 
Indeed, assume that for some 
$\tau$, $\sum_{x\in X}{\sf s}_{\tau}x\neq 1$.
Let $a=1-\sum_{x\in X} {\sf s}_{\tau}x$. 
Then $a\neq 0$. But then there exists $x\in X$, such that ${\sf s}_{\tau}x\cdot 
\sum_{x\in X}{\sf s}_{\tau}x=0$
which is impossible. 
\end{demo}
We do not need all the $\psi_{\tau}$ for if $\tau$ is a bijection then ${\sf s}_{\tau}$ is self - conjugate,
hence completely additive.

The question that imposes itself now, is whether there are atomic polyadic algebras that are not completely additive. We do not know.
But here we give an example taken from \cite{AGNS}, and slightly modified, to show that if we 
restrict substitutions to only injective ones, together with all replacements and ones whose kernels induce finite equivalence classes, 
then the resulting variety is not completely additive.

However, the algebra that witnesses the non-complete additivity of the substitution operator 
corresponding to one replacement is {\it not} atomic.
Notice that here the resulting semigroup, generated by the available substitutions, 
has the same cardinality as $|\alpha|$, 
so that we are discarding quite a few substitutions, and this is a significant change.

But we believe that the non complete additivity of replacements
(even in this restricted context), does give an indication that there are (full) polyadic agebras that are not completely additive; but nevertheless
atomicity could well be a prohibiting factor to non-complete additivity of substitutions corresponding to non-bijective maps.

However, we also show here that {\it atomic} polyadic algebras of dimension $2$ may not be completely representable 
(The class of completely representable algebras, in this case, is also elementary. This can be proved exactly as \ref{r} above.
For higher finite dimensions the class of completely representable algebras is not even elementary \cite{Moh1}. This is a non-trivial result and the proof
in \cite{Moh1} uses a rainbow construction.)

\begin{example} 

\begin{enumarab}

\item Let $\B$ be an atomless Boolean algebra that has a Stone representation with unit $U$ such that for any distinct $u,v\in U$, 
there is $X\in B$ such that $u\in X$ and $v\in {}\sim X$.
Let $\alpha$ be an infinite ordinal. Let $R=\{\prod X_i : i\in \alpha, X_i\in \B\}\subseteq {}^{\alpha}U$.
and $A=\{\bigcup S: S\subseteq R: |S|<\omega\}.$
Then $A$ is the base of a subalgebra of $\wp(^{\alpha}U)$, call this algebra $\A$.
This  can be proved exactly as in \cite{AGNS} 
by noting that substitutions corresponding to injections just permute components, and that 
for $\Gamma\subseteq \alpha$, ${\sf C}_{(\Gamma)}R$ is the element in $A$ 
that agrees with $R$ off of $\Gamma$, that is it is the same as $R$ in all components, 
except for $i\in \Gamma$; here the $i$ th component is blown up to $U$.

Let $S=\{X\times \sim X\times U \times U\times\ldots : X\in B\}.$
Then, like in \cite{AGNS}, we have  $S_0^1(\sum S)= {}^{\alpha}U$
and $\sum \{S_{0}^1(Z): Z\in S\}=\emptyset.$

If $\tau\in {}^{\alpha}\alpha$ is not bijective, then we face a problem; the algebra $\A$ defined above might not be closed under such a $\tau$. 
Let $\P$ be the partition induced by  $ker(\tau)=\{(i,j)\in \alpha\times \alpha: \tau(i)=\tau(j)\}$,
and write $i\sim j$ if $\tau(i)=\tau(j)$.
Assume that $\P=\{P_n: n\in I\}$, where $|I|\leq |\alpha|$.
Let $i_n=min P_n$, and assume that for $n<m$, $i_n< i_m$. Let $J=I\sim \{i_n: n\in I\}$.
For $i\in I$, let $G_{i_n}=\bigcap_{i\in \alpha, i\sim i_n}X_i$, and $G_j=U$ if $j\notin J$.
Then ${\sf s}_{\tau}\prod X_i= \prod_{m\in I} G_{m}.$  
This last might not be in the algebra because it involves possibly infinite intersections.

And even if we take a {\it complete} 
atomless Boolean algebra, which exists, 
then we know that this algebra {\it cannot} be completely representable, for a complete representation entails that the algebra is 
atomic which is not the case. 
So though arbitrary meets exist in the algebra (by completeness) they are not necessarily 
reflected by infinite interestions in
the Stone representation. In other words, given $X_i: i\in I$, $I$ an infinite set, and $X_i\in \B$, there is nothing to guarantee that
$\bigcap_{i\in i} X_i$ is in $\B$. We find that this example is a near miss, 
and we conjecture that it can be appropriately modified to give 
a polyadic algebra of infinite dimension that is not 
completely additive. But this reasoning also tells us that we can count in those not necessarily injective maps 
whose kernels give finite equivalence classes, each such class renders only a finite intersection.

\item We show that atomic polyadic algebras of dimension $2$ 
may not be completely representable. For every infinite cardinal $\kappa$ we construct such an algebra with cardinality
$\kappa$. By our theorem \ref{r}, 
it suffices to show that one of the operations is not completely additive. The example is
also from \cite{AGNS}. We give the outline. Let $|U|=\mu$ be an infinite set and $|I|=\kappa$ be a cardinal such 
that $Q_n$, $n\in \kappa$,  is a family of relations that form a partition of $U\times U$, 
Let $i\in I$, and let $J=I\sim \{i\}$. Then of course $|I|=|J|$. Assume that $Q_i=D_{01}=\{s\in V: s_0=s_1\},$
and that each $Q_n$ is symmetric; that is for any  $S_{[0,1]}Q_n=Q_n$ and furthermore, that $dom Q_n=range Q_n=U$ for every $n\in \kappa$.
It is straightforward to show that such partitions exist.


Now fix $F$ a non-principal ultrafilter on $J$, that is $F\subseteq \mathcal{P}(J)$. 
For each $X\subseteq J$, define
\[
 R_X =
  \begin{cases}
   \bigcup \{Q_k: k\in X\} & \text { if }X\notin F, \\
   \bigcup \{Q_k: k\in X\cup \{i\}\}      &  \text { if } X\in F
  \end{cases}
\]

Let $$\A=\{R_X: X\subseteq I\sim \{i\}\}.$$
Notice that $|\A|\geq \kappa$. Also $\A$ is an atomic set algebra with unit $R_{J}$, and its atoms are $R_{\{k\}}=Q_k$ for $k\in J$.
(Since $F$ is non-principal, so $\{k\}\notin F$ for every $k$). This can be proved exactly like in \cite{AGNS}. 
The subalgebra $\B$ generated by the atoms is as required. 
We should also mention that this example shows that Pinters algebra, which are cylindrfier free algebras of polyadic algebras
for all dimensions may not be completely representable 
answering an implicit question of Hodkinson's \cite{h} top of p. 260. 
(In the absence of cylindrifiers the construction lifts to arbitrary dimensions, because we do not require that 
$dom Q_n =range Q_n=U$.) For infinite dimensions weak set algebras can be used. A weak set algebra 
is one whose unit is the set of sequences agreeing 
co-finitely with a given one.

\end{enumarab}

\end{example}\label{ex}

\begin{example} For cylindric algebras minimal dilations of atomic algebras may not be atomic. This is quite easy to show.
Let $\A\in \RCA_n$ such that $\A$ is atomic. Let $\B\in \CA_{\omega}$ such that $\A=\Nr_n\B$. Obviously such algebras exist.
Let $\B'=\Sg^{\B}A$, then $\B'$ is locally finite, and $\A=\Nr_n\B'$. 
Locally finite algebras are not atomic. Another even easier example is that if one takes a simple locally finite non-atomic algebra $\A$, then
$\Nr_0\A=\{0,1\}$. In fact, for cylindric algebras minimal dilations 
may not be unique, so that unlike polyadic algebras, we cannot speak about {\it the } minimal dilation 
of an even representable algebra. This property,  is strongly linked to the amalgamation property for the class of representable cylindric algebras, 
or rather, the lack of \cite{Sayedneat}. 
\end{example}

\begin{question} Are there atomic polyadic algebras that are not completely additive, hence not completey representable?
\end{question}

Dedekend or MacNeille completions for $\PA$, which we refer to as minimal completions,
are also problematic, and they have to be re-defined to adapt the possibility of non-complete additivity.

A definition of completions of not necessarily completely additive varieties of Boolean algebras with operators is given in \cite{AGNS}, 
but to our mind it is not  satisfactory, for it gives, for example, that the completion of the atomic algebra $\A$ 
in the second item of the previous example, it itself, for this algebra is complete.
And so the completon is {\it not} completely additive, the completion of $\A$ is not $\Cm\At\A$.

However, if $\A\in \PA_{\alpha}$ happens to be completely additive, then it has a minimal completion, 
because the equations axiomatizing $\PA_{\alpha}$ are Sahlqvist, and these are preserved under taking minimal completions of a completely additive algebra.. 
We do not know whether $\PA_{\alpha}$  is atom-canonical, nor even single-persistent. 
That is, if  $\A$ is atomic and {\it not} completely additive, is the complex algebra of its atom structure a polyadic algebra, in symbols, 
$\Cm\At\A\in \PA_{\alpha}$?  Is the algebra generated by the singletions of $\A\A$ a $\PA_{\alpha}$?, not that an affirmative answer to the first question implies an affirmative answer to the second,
but the converse is not true.

Note that in this case, since $\Cm\At\A$ is completely additive (complex algebras are completely additive) 
and $\A$ is not, so that $\A$  does not embed into $\Cm\At\A$ via $a\mapsto \{x\in \At\A: x\leq a\}$.
If $\A$ is completely additive and atomic, then the complex algebra of its atom structure, namely, $\Cm\At\A$  is just its ordinary 
completion, with the above embedding.

\begin{question} Is $\PA_{\alpha}$ atom canonical?
\end{question}

On the other hand, $\PA_{\alpha}$ is a canonical variety because again it is axiomatized by Sahlqvist equations (in fact, positive ones). 
Furthermore, if $\A\in \PA_{\alpha},$ then its canonical extension is $\Nr_{\alpha}\B^+$ where $\B^+$ is the complex algebra of the minimal dilation 
$\B$ of $\A$ in $\geq \omega$ many dimensions ($\B^+$ is unique, it does not depend on the number of extra dimensions). 

\section{ A metalogical interpretation in Keisler's logic}

Polyadic algebras of infinite dimension correspond 
to a certain infinitary logic studied by Keisler, and referred to in the literature as Keisler's logic. Keisler's 
logic allows formulas of infinite length and quantification on infinitely many 
variables, but is does not allow infinite conjunctions, with semantics defined as expected.
While Keisler \cite{K}, and independently Monk and Daigneault \cite{DM}, 
proved a completeness theorem for such logics, our result implies a 
`Vaught's theorem' for such logics, namely, that certain atomic theories, namely those whose Tarski Lindenbaum algebra is completely additive,
have atomic models, in a sense to be made precise.

Let $\L$ denote Keislers's logic with $\alpha$ many variables 
($\alpha$ an infinite ordinal). For a structure $\M$, a formula $\phi$,  and an assignment $s\in {}^{\alpha}M$, we write
$\M\models \phi[s]$ if $s$ satisfies $\phi$ in $\M$. We write $\phi^{\M}$  for the set of all assignments satisfying $\phi.$

\begin{definition} Let $T$ be a given $\L$ theory. 
\begin{enumarab}
\item A formula $\phi$ is said to be complete in $T$ iff for every formula $\psi$ exactly one of 
$$T\models \phi\to \psi, \\ T\models \phi\to \neg \psi$$
holds.
\item A formula $\theta$ is completable in $T$ iff there there is a complete formula $\phi$ with $T\models \phi\to \theta$.
\item $T$ is atomic iff if every formula consistent with $T$ is completable in $T.$

\item A model $\M$ of $T$ is atomic if for every $s\in {}^{\alpha}M$, there is a complete formula $\phi$ such that $\M\models \phi[s].$
\end{enumarab}
\end{definition}
We denote the set of formulas in a given language by $\Fm$ and for a set of formula $\Sigma$ we write $\Fm_{\Sigma}$ for the Tarski-
Lindenbaum quotient (polyadic)
algebra. 

\begin{theorem} Let $T$ be an atomic theory in $\L$ and assume that $\phi$ is consistent with $T$. Assume futher that 
$\Fm_T$ is completely additive. Then $T$ has an atomic model in which $\phi$ is satisfiable.
\end{theorem}
\begin{demo}{Proof}  Assume that $T$ and $\phi$ are given. Form the Lindenbaum Tarski algebra $\A=\Fm_T$ and let $a=\phi/T$.
Then $\A$ is an atomic completely aditive polyadic algebra, since $T$ is atomic, and $a$ is non-zero, because $\phi$ is consistent with $T$.
Let $\B$ be a set algebra with base $M$, and $f:\A\to \B$ be a complete representation
such that $f(a)\neq 0$.
We extract a model $\M$ of $T$, with base $M$, from $\B$ as follows. 
For  a relation symbol $R$ and $s\in {}^{\alpha}M$, $s$ satisfies $R$ if $s\in h(R(x_0,x_1\ldots ..)/T)$. Here the variables occur in their 
natural order.
Then one can prove by a straightforward induction that $\phi^{\M}=h(\phi/T)$. Clearly $\phi$ is satisfiable in $\M$. 
Moreover, since the representation 
is complete it readily follows that  $\bigcup\{\phi^{\M}: \phi\text { is complete }\}={}^{\alpha}M$, and we are done.
\end{demo}

If we add infinite conjunctions to our logic and stipulate that for any theory 
$T\vdash {\sf s}_{\tau}\bigwedge \phi_T\equiv \bigwedge {\sf s}_{\tau}\phi_T$, then we get a proper extension of Keislers logic 
such that Tarski Lindenbaum algebras are completely 
additive, and so in this case atomic theories with no extra conditions will have atomic models.
One way to do this is to stipulate the axiom $\vdash {\sf s}_{\tau}\bigwedge \phi\equiv \bigwedge {\sf s}_{\tau}\phi.$
Then for such expansion of Keiler's logic with infinite conjunction, every atomic theory with no further conditions has an atomic model.

Results in algebraic logic are more interesting when they have immediate impact on the logic side be it first order logic or extensions thereof.
For ordinary first order logic atomic theories  in countable languages have atomic models, 
as indeed Vaught proved, but in the first order context countability is essentially 
needed.  Furthermore, in the context of first order logic, atomic countable models for atomic theories are unique (up to isomorphism).

Our theorem can also be regarded as an omitting types theorem, for possibly uncountable languages, for the representation constructed in our theorem
omits the set (or infinitary type) of co-atoms in the sense that the representation $f$, defined in our main theorem, satisfies 
$\bigcap_{x\in X^{-}}f(x)=\emptyset$.  

A standard omitting types theorem for Keisler's logic, addressing the omission of 
a family of types, not just one, is highly 
problematic since, even in the countable case, i.e when the base of the algebra considered  is countable, since we have uncountably many 
operations.  Nevertheless, a  natural omitting types theorem can be formulated as follows. 

Let $\L$ denote Keisler's logic, and let $T$ be an $\L$ theory. A set $\Gamma\subseteq \Fm$ is principal, 
if there exist a formula $\phi$ consistent with $T$,
such that $T\models \phi\to \psi$ for all $\psi\in \Gamma$. Otherwise $\Gamma$ is non-principal. A model $\M$ of $T$ omits $\Gamma$, if
$\bigcap_{\phi\in \Gamma}\phi^{\M}=\emptyset$, where $\phi^{\M}$ is the set of assignments satisfying $\phi$ in $\M$.  
Then the omitting types theorem in this context says:
If $\Gamma$ is non-principal, then there is a model $\M$ of $T$ that omits $\Gamma$.
Algebraically:

\begin{athm}{OTT} Let $\A\in \PA_{\alpha}$ be completely additive and $a\in A$ be non-zero. Assume that $X\subseteq A$, is such that $\prod X=0$.
Then there exists a set algebra $\B$ 
and a homomorphism $f:\A\to \B$ such that $f(a)\neq 0$ and $\bigcap_{x\in X}f(x)=\emptyset.$
\end{athm}

Unlike omitting types theorems for countable languages, the proof cannot resort to the Baire Category theorem, 
for the simple reason that the Baire category theoem applies only to the countable case.
Nevertheless, our proof  of theorem \ref{main}, shows how to omit  the non principal type consisting soley of co-atoms, 
basically because the set principal ultrafilters is dense in the Stone topology,  and a principal ultrafilter lie outside nowhere dense sets. 
It is not at all clear how to omit arbitrary non-principal types, when the algebra in question is not atomic. 
If we take only the set of infinitary joins corresponding to quantifier  elimination, then an ultrafilter 
preserving them can be found giving an {\it ordinary} representation \cite{DM}, but there is no topology involved here, at least in the proof of Diagneault and
Monk; the ultrafilter is built in a step by step fashion. 
In case our algebra is completely additive, then we would have a second infinitary meet that we want to omit, 
and the corresponding set
(which is now an infinite intersection) in the Stone space is also nowhere dense.

But in this case, when this meet is not the set of co-atoms, example when the algebra is atomless,
we do not guarantee that such a set consisting of the ultrafilters (models) {\it not } omitting 
the non principal  type do not exhaust the set consisting of those ultrafilters preserving only cylindrifier elimination, these are the ultrafilters from which we 
obtain representations.
(This cannot happen in the countable case where the Baire Category theorems entails that the complement of such a set or even a 
complement of $< covK$ union of such sets 
is dense. Here $covK$ is the least cardinal such that the Baire Category theorem 
for compact Hausdorff spaces fail and also the largest cardinal for which Martin's axiom on countable Boolean algebra holds, 
and it is the best estimate for number of non isolated types 
omitted.) 

But in this form of generality the omitting types theorem fails as the next easy example show:

\begin{example} Let $T$ be a an uncountable complete first order theory {\it without} equality, in an uncountable languages 
having a sequence of variables of order type
$\omega$. Assume that  there exists a non principal type $\Gamma$ of this theory that cannot be omitted.
Easy examples are known. Then take the {\it locally finite polyadic Tarski Lindenbaum algebra} $\A\in \PA_{\omega}$ 
based on this theory, and let $X=\{\phi_T: \phi\in \Gamma\}$. Then $\prod X=0$ and there is no (locally finite) 
polyadic set algebra omitting this meet.
\end{example}

It is worthy of note that locally finite algebras, except in trivial cases, are {\it atomless}.

Nevertheless, there is yet another interesting  connection between complete representations and omitting types. 
More can be said here. A classical theorem of Vaught for first order logic says that countable atomic theories have countable atomic models,
such models are necessarily prime, and a prime model omits all non principal types.
We have an analogous situation here, and the proof is very simple, assuming simplicity (in the universal algebraic sense) of our algebra,
that is, assuming that the corresponding theory in Keisler's logic is complete. The general case is not much harder, we just work with disjoint unions
of square units.

\begin{theorem} Let $f:\A\to \wp(^{\alpha}U)$ be a complete representation of $\A\in \PA_{\alpha}$.
Then for any set $I$, for any given family $(Y_i:i\in I)$ of subsets of $\A$,  if $\prod Y_i=0$ for all $i\in I$, then we have
$\bigcap_{y\in Y_i} f(y)=\emptyset$ for all $i\in I$.
\end{theorem}
\begin{proof} Let $i\in I$. Let $Z_i=\{-y: y\in Y_i\}$. Then $\sum Z_i=1$. $\A$ is completely representable, hence
$\A$ is atomic, and so for any atom $x$, we have $x.\sum Z_i=x\neq 0$.
Hence there exists $z\in Z_i$, such that
$x.z\neq 0$. But $x$ is an atom, hence $x.z=x$ and so $x\leq z$. We have shown that for every atom $x$, there exists $z\in Z_i$ such that $x\leq z$.
It follows immediately, since a complete representation is an atomic one,  that
$^{\alpha}U=\bigcup_{x\in \At\A}f(x)\leq \bigcup_{z\in Z_i} f(z)$, 
and so, $\bigcap_{y\in Y_i} f(y)=\emptyset,$
and we are done.
\end{proof}

Every principal ultrafilter in a completely additive polyadic algebra
gives rise to a complete representation (an atomic model). Are these representations, like first order logic, isomorphic, 
that is, are the generalized
set algebras (obtained by taking the disjoint union of the bases over non-zero elements of the algebra) base isomorphic? 

This is not case because the base of our algebras can be any of any cardinality $\geq$ $\mathfrak{n}$ with $\mathfrak{n}$
as in the proof of theorem \ref{main}, so that set algebras constructed {\it cannot} 
be base isomorphic.   It is true that the dilation is unique in a fixed dimension, but if we take a larger dimension, 
we get another also unique dilation but in this 
larger dimension. Evidently the two dilations cannot be isomorphic for the very simple reason that they have different similarity types.
This is a significant deviation fom first order logic.

Let us formulate this last paragraph in a theorem. Let $\A\in \PA_{\alpha}$ be  simple,  
infinite and hereditory atomic (every subalgebra is atomic). assume that $|A|=|\alpha|.$
Then the the number of principal ultrafilters $\leq|\alpha|$, 
but if it not hereditory atomic then this number  is
$\leq {}^{|\alpha|}2$. However, regardless of the number of ultrafilters in the Stone space, we have: 

\begin{theorem} Let $T$ be a complete atomic theory in Keisler's logic such that $\A=\Fm_T$ is completely additive. 
Let $\mathfrak{m}$ be the local degree of $\A$, $\mathfrak{c}$ its effective cardinality and 
$\mathfrak{n}$ be any cardinal such that $\mathfrak{n}\geq \mathfrak{c}$ 
and $\sum_{s<\mathfrak{m}}\mathfrak{n}^s=\mathfrak{n}.$ Then $T$ has an atomic model of size $\mathfrak{n}.$
\end{theorem}
\begin{proof} See the proof of theorem  \ref{main}. Here the representing function is injective, due to simplicity of $\Fm_T$ inducing an isomorphism.
\end{proof}

So here each principal ultrafilter gives rise to infinitely many atomic representations.
and we are infront of an Upward Skolem Theorem addressing atomic models, 
for every atomic model there is one with larger cardinality.  
A natural question here is that if $\mathfrak{n}< \mathfrak{m}$, is there perhaps a {\it subbase isomorphism} between 
the representation corresponding to adding $\mathfrak{n}$ witnesses,  into that corresponding to adding $\mathfrak{m}$ witnesses,
which is an algebraic reflection of an  elementary embedding?

We should mention that  an analogous result holds for several reducts of polyadic algebras (without equality), namely,
complete additivity and atomicity is equivalent to complete representability. 
For example if we take the reduct by restricting 
cylindrifiers on only those subsets $\Gamma$ of $\alpha$ such that 
$|\Gamma|< \kappa\leq |\alpha|$, where $\kappa$ is an infinite cardinal, then we get the same result. 

The point is, as long as we have {\it all substitutions}, then we have rich transformation systems,
hence functional dilations, from which we can get dilations of the algebra in question by discarding cylindrifiers, forming the functional dilation using
substitutions, then re-defining cylindrifiers (to agree with the old ones in the neat reduct) as we did in our proof of \ref{main}. 
Furthermore, in all cases if the original algebra is atomic, then so will be the dilation 
(because it is basically a product of the atomic boolean algebra), 
hence the proof survives verbatim.

\subsection{Modified Keisler's logic with equality}

The case of polyadic agebras of infinite dimension {\it with equality} is much more involved. 
In this case the class of representable algebras is not a variety; it is not closed under ultraproducts, 
although every algebra has the neat embedding property (can be embedded into the neat reduct of algebras in every higher 
dimension).  In particular, we do not guarantee that atomic algebras are even representable, let alone admit a complete representation. 
Still we can ask whether atomic {\it representable} algebras are completey representable. 
The question seems to be a hard one, because we cannot resort to a neat embedding theorem as we did here for the equality free case.

In fact, finding neat embedding theorems for polyadic equality algebras, that enforce  even relativized representations is a 
very active topic, that is gaining a lot of momentum \cite{1}.
One can find well motivated appropriate notions of relativized semantics by first locating them while giving up classical semantical prejudices. 
It is hard to give a precise mathematical underpinning to such intuitions. What really counts at the end of the day
is a completeness theorem stating a natural fit between chosen intuitive concrete-enough, but perhaps not {\it excessively} concrete, 
semantics and well behaved axiomatizations. 
The move of altering semantics has radical phiosophical repercussions, taking us away from the conventional 
Tarskian semantics captured by Fregean-Godel-like axiomatization; the latter completeness proof is effective 
but highly undecidable; in modal logic and finite variable fragments of first order logic, which have a modal formalism, this is 
highly undesirable.

Now we show that such algebras, when atomic,
admit complete {\it relativized} representations.

Ferenczi has a lot of work in this direction \cite{f}. In the latter article, 
he defines an abstract equational class $\CPEA_{\alpha}$, $\alpha$ an infinite ordinal,  which is like 
cylindric algebras in that it has only finite cylindrifiers and like polyadic algebras, in that its cylindrifier free reduct,
forms a transformation system; all substitutions are available. Also, and this is the most important part this class has diagonal elements.
The presence of diagonal elements makes the variety completely additive, which is an acet, in our context seeking complete representations.
However,  full fledged commutativity of cylindrifiers do not hold here. 

Ferenczi proves a {\it strong completeness} theorem in analogy the Diagneault Monk representation theorem for polyadic algebras, namely, 
$\CPEA_{\alpha}=\Gp_{\alpha}$, where $\Gp_{\alpha}$ is a class of set algebras whose units are relativized. (It is a union of weak set algebras 
that are not necessarily disjoint). 
The technique is a Henkin construction implemented via a neat 
embedding theorem. This is definitely an achievement, 
because for classical polyadic {\it equality} algebras, when cylindrifiers commute,  
neat embeddability into infinitely many extra dimensions
{\it does not} enforce representability. 

But the choice of representing ultrafilters (which Ferenczi calls perfect) 
here is more delicate, and the representation
is somewhat more intricate than the classical case. 
This is due to the fact that cylindrifiers and substitutions do not commute, in cases where it is consistent that they do as witnessed by classical 
representations. Sacrifizing commutativity of cylindrfiers and for that matter commutativity of cylindrifiers and substitutions
make relativized representability possible. Also the representant class of relativized algebras turns out to be a variety; 
this is not the case with classical representations for polyadic equality algebras having square Tarskian semantics, even if we restrict the similarity type to only finite
cylindrifiers. In the latter case the presence of diagonal elements together with infinitary substitutions
make this class not closed under ultraproducts.

But such modifications will survive our proof.
And indeed we can show also using a neat embedding theorem together with our 
previous topological argument, baring in mind that we can omit the condition of complete additivity since it holds anyway, we have:
(However, we give the general idea and some of the details will be omitted, but can be recovered easily from \cite{f}.)
\begin{theorem} Every atomic $\CPEA_{\alpha}$ is completely representable
\end{theorem}
\begin{proof}

Let $c\in \A$ be non-zero. We will find a $\B\in \Gp_{\alpha}$ and a homomorphism 
from $f:\A \to \B$ that preserves arbitrary suprema
whenever they exist and also satisfies that  $f(c)\neq 0$. 
Now there exists $\B\in \CPEA_{\mathfrak{n}}$, $\mathfrak{n}$ a regular cardinal.
such that $\A\subseteq \Nr_{\alpha}\B$ and $A$ generates $\B$. 
Note that  $|\mathfrak{n}\sim \alpha|=|\mathfrak{n}|$ and
for all $Y\subseteq A$, we have $\Sg^{\A}Y=\Nr_{\alpha}\Sg^{\B}Y.$ 
This dilation also has Boolean reduct isomophic to $F({}^\mathfrak{n}\alpha, \A)$, in particular, it is atomic because $\A$ is atomic.
Also cylindrifiers are defined on this minimal functional dilation excatly like above by restricting to singletions.
Let $adm$ be  the set of admissable substitutions. $\tau\in \B$ is admissable if 
$Do\tau\subseteq \alpha$ and $Rg\tau\cap \alpha=\emptyset$. 
Then we have
for all $i< \mathfrak{n}$ and $\sigma\in adm$,
\begin{equation}\label{tarek1}
\begin{split}
s_{\sigma}{\sf c}_{i}p=\sum s_{\sigma}{\sf s}_i^jp 
\end{split}
\end{equation}
This uses that ${\sf c}_k=\sum {\sf s}_k^i x$, which is proved like the cylindric case; the proof depends on diagonal elements.
Let $X$ be the set of atoms of $\A$. Since $\A$ is atomic, then  $\sum^{\A} X=1$. By $\A=\Nr_{\alpha}\B$, we also have $\sum^{\B}X=1$.
Because substitutions are completely additive we have
for all $\tau\in {}^{\alpha}\mathfrak{n}$
\begin{equation}\label{tarek2}
\begin{split}
\sum {\sf s}_{\bar{\tau}}^{\B}X=1.
\end{split}
\end{equation}
Let $S$ be the Stone space of $\B$, whose underlying set consists of all boolean ulltrafilters of 
$\B$, and let $F$ be a principal ultrafilter chosen as before.
Let $\B'$ be the minimal completion of $\B$. Exists by completey additivity. Take the filter $G$ in $\B'$ generated by the generator of $F$
and let $F=G\cap \B$. Then $F$ is a perfect ultrafilter. Because our algebras have diagonal algebras, 
we have to factor our base by a congruence relation that reflects
equality.
Define an equivalence relation on $\Gamma=\{i\in \beta:\exists j\in \alpha: {\sf c}_i{\sf d}_{ij}\in F\}$,
via $m\sim n$ iff ${\sf d}_{mn}\in F.$ Then $\Gamma\subset \alpha$ and 
the desired representation is defined on a $\Gp_{\alpha}$ with base
$\Gamma/\sim$. We omit the details.
\end{proof}

The metalogical interpretation of the above theorem is also interesting. 
It gives  Vaught's theorem for a variant of Keisler's logic {\it with equality}, which we call modified Keisler's logic.
This variant is defined by taking only finite cylindrifiers, and weakening the axioms (concerning commutativity of the 
various non boolean operations). This calculas is complete, 
but with respect to relativized semantics. This is proved by Ferenczi. 
Here our theorem says that {\it any} atomic theory as defined above, for Keislers's logic, 
adapted to the present context, has an atomic model. We do not need to add 
the algebraic condition of additivity of
the formula algebra, {\it it is} completely additive. This is a more elegant formulation; it does not refer to an algebraic property that has only 
a vague logical counterpart. In fact the modified Keisler's logics share quite a few properties with first order logic, as we proceed to show next.
Let $\L^{=}$ be the modified Keislers logic. Then 

\begin{theorem}
\begin{enumarab}
\item $\L^{=}$ is strongly complete with respect to relativized semantics. That is, for any set of formulas $\Gamma\cup \{\phi\},$ if $\Gamma\models \phi$, 
then $\Gamma\vdash \phi.$ 
\item $\L^{=}$ has the interpolation property, hence Beth definability
\item  If $T$ be a complete atomic theory in the modified Keisler's logic with equality.
Let $\mathfrak{m}$ be a regular cardinal such that $|\alpha|<\mathfrak{m}$. Then $T$ has an atomic model of size $\mathfrak{m}.$
\item If $T$ is an atomic complete theory, then $T$ omits all non-principal types
\end{enumarab}
\end{theorem}
\begin{proof} The first part is due to Ferenczi. The third part follows from the fact that
principal ultrafilters respecting the given set of infinitary joins (corresponding to 
substituted versions of co atoms and elimination of 
cylindrifiers) always exist in the Stone space of the atomic  
dilation in $\mathfrak{m}$ extra dimensions, when the latter is a regular cardinal $>|\alpha|$. The fourth part follows from the argument used above for 
Keisler's logic.

For the second part we give only a sketch. The full proof will be submitted elsewhere.
Let $\beta$ be a cardinal, and $\A=\Fr_{\beta}{}\CPEA_{\alpha}$ be the free algebra on $\beta$ generators.
Let $X_1, X_2\subseteq \beta$, $a\in \Sg^{\A}X_1$ and $c\in \Sg^{\A}X_2$ be such that $a\leq c$.
We show that there exists $b\in \Sg^{\A}(X_1\cap X_2)$ such that $a\leq b\leq c.$ This is the algebraic version of the Craig 
interpolation property.
Assume that $\mu$ is a regular cardinal $>max(|\alpha|,|A|)$.
Let $\B\in \CPEA_{\mu}$, such that $\A=\Nr_{\alpha}\B$,
and $A$ generates $\B$. Such dilations exist.
Ultrafilters in dilations used to represent algebras in $\CPEA$ are  as before 
defined via  the {\it admitted substitutions}, which
we denote by $adm.$ Recall that very admitted substitution has a domain $Do\tau$ which is  subsets of $\alpha$ and a range, 
$Rg\tau$ such that $Rg\tau\cap \alpha=\emptyset$.
One defines special  filters in the dilations $\Sg^{\B}X_1$ and in $\Sg^{\B}X_2$
like but they have to be compatible on the common subalgebra. This needs some work.
Assume that no interpolant exists in $\Sg^{\A}(X_1\cap X_2)$.
Then no interpolant exists in $\Sg^{\B}(X_1\cap X_2)$.
We eventually arrive at a contradiction.
Arrange $adm\times\mu \times \Sg^{\B}(X_1)$
and $adm\times\mu\times \Sg^{\B}(X_2)$
into $\mu$-termed sequences:
$$\langle (\tau_i,k_i,x_i): i\in \mu\rangle\text {  and  }\langle (\sigma_i,l_i,y_i):i\in \mu\rangle
\text {  respectively.}$$
is as desired.
Thus we can define by recursion (or step-by-step)
$\mu$-termed sequences of witnesses:
$$\langle u_i:i\in \mu\rangle \text { and }\langle v_i:i\in \mu\rangle$$
such that for all $i\in \mu$ we have:
$$u_i\in \mu\smallsetminus
(\Delta a\cup \Delta c)\cup \cup_{j\leq i}(\Delta x_j\cup \Delta y_j\cup Do\tau_j\cup Rg\tau_j\cup Do\sigma_j\cup Rg\sigma_j)\cup \{u_j:j<i\}\cup \{v_j:j<i\}$$
and
$$v_i\in \mu\smallsetminus(\Delta a\cup \Delta c)\cup
\cup_{j\leq i}(\Delta x_j\cup \Delta y_j\cup Do\tau_j\cup Rg\tau_j\cup Do\sigma_j\cup Rg\sigma_j))\cup \{u_j:j\leq i\}\cup \{v_j:j<i\}.$$
For an  algebra $\cal D$ we write $Bl\cal D$ to denote its boolean reduct.
For $i, j<\mu$, $i\neq j$,
$s_i^jx=c_i(d_{ij}\cdot x)$ and $s_i^ix=x$. $s_i^j$ is a unary operation
that abstracts the operation of substituting the
variable $v_i$ for the variable $v_j$ such that
the substitution is free.
For a boolean algebra $\cal C$  and $Y\subseteq \cal C$, we write
$fl^{\cal C}Y$ to denote the boolean filter generated by $Y$ in $\cal C.$  Now let
$$Y_1= \{a\}\cup \{-{\sf s}_{\tau_i}{\sf  c}_{k_i}x_i+{\sf s}_{\tau_i}{\sf s}_{u_i}^{k_i}x_i: i\in \mu\},$$
$$Y_2=\{-c\}\cup \{-{\sf s}_{\sigma_i}{\sf  c}_{l_i}y_i+{\sf s}_{\sigma_i}{\sf s}_{v_i}^{l_i}y_i:i\in \mu\},$$
$$H_1= fl^{Bl\Sg^{B}(X_1)}Y_1,\  H_2=fl^{Bl\Sg^B(X_2)}Y_2,$$ and
$$H=fl^{Bl\Sg^{B}(X_1\cap X_2)}[(H_1\cap \Sg^{B}(X_1\cap X_2)
\cup (H_2\cap \Sg^{B}(X_1\cap X_2)].$$
Then $H$ is a proper filter of $\Sg^{B}(X_1\cap X_2)$. This can be proved by a tedious induction, with the base provided
by the fact that no interpolant exists in the dilation.
Proving that $H$ is a proper filter of $\Sg^{\B}(X_1\cap X_2)$,
let $H^*$ be a (proper boolean) ultrafilter of $\Sg^{\B}(X_1\cap X_2)$
containing $H.$
We obtain  ultrafilters $F_1$ and $F_2$ of $\Sg^{\B}(X_1)$ and $\Sg^{\B}(X_2)$,
respectively, such that
$$H^*\subseteq F_1,\ \  H^*\subseteq F_2$$
and (**)
$$F_1\cap \Sg^{\B}(X_1\cap X_2)= H^*= F_2\cap \Sg^{\B}(X_1\cap X_2).$$
Now for all $x\in \Sg^{\cal B}(X_1\cap X_2)$ we have
$$x\in F_1\text { if and only if } x\in F_2.$$
Also from how we defined our ultrafilters, $F_i$ for $i\in \{1,2\}$ are {\it perfect}, a term introduced by Ferenczi.
Then one defines  homomorphisms, one on each subalgebra, like in \cite{Sayedneat} p. 128-129, using the
perfect ultrafilters to define a congruence relation on $\beta$ so that the defined homomorphisms respect diagonal elements.
Then freeness will enable paste these homomophisms, to a single one defined to the set of free generators,
which we can assume to be, without any loss, to
be $X_1\cap X_2$ and it will satisfy  $h(a.-c)\neq 0$ which is a contradiction.

\end{proof}

The notion of relativized representations constitute a huge topic in both algebraic and modal logic, see the introduction of \cite{1},  \cite{f}, \cite{v}.
Historically,  in \cite{HMT2} square units got all the attention and relativization  was treated as a side issue.
However, extending original classes of models for logics to manipulate their properties is common. 
This is no mere tactical opportunism, general models just do the right thing.

The famous move from standard models to generalized models is 
Henkin's turning round  second  order logic into an axiomatizable two sorted first
order logic. Such moves are most attractive 
when they get an independent motivation. 

The idea is that we want to find a semantics that gives just the right action 
while additional effects of square set theoretic representations are separated out as negotiable decisions of formulation 
that can threaten completeness, decidability, and interpolation. 
(This comes across very much in cylindric algebras, especially in finite variable fragments of first order logic,
and classical polyadic equality algebras, in the context of Keisler's logic with equality \cite{Sagi}.)

Using relativized representations Ferenczi \cite{f}, proved that if we weaken commutativity of cylindrifiers
and allow  relativized representations, then we get a finitely axiomatizable variety of representable 
quasi-polyadic equality algebras (analogous to the Andr\'eka-Resek Thompson ${\sf CA}$ version, cf. \cite{Sayedneat} and \cite{f}, 
for a discussion of the Andr\'eka-Resek Thompson breakthrough for cylindric-like algebras); 
even more this can be done without the merry go round identities.
This is in sharp contrast with the long list of  complexity results proved for the commutative case \cite{f}.
Ferenczi's results can be seen as establishing a hitherto fruitful contact between neat embedding theorems 
and relativized representations, with enriching 
repercussions and insights for both notions.

Now coming back to the classical case, where we have full fledged commutativity of cylindrifiers and classical Tarskian square semantics, 
if we  restrict the signature of $\PA_{\alpha}$ to only these  substitutions whose support has cardinality $<\kappa$,
where $\kappa\leq |\alpha|,$ 
then we conjecture  that in this new signature 
atomic completely additive algebras, in the classical sense, 
may not be completely reprresentable, not only that, but we further conjecture that the class of completely representable 
algebras may turn out non-elementary. Here the support of a map $\tau$ on $\alpha$ is $\{i\in \alpha: i\neq \tau(i)\}$.
As a matter of fact, if we have  a single diagonal element then indeed this will be  the case, as shown below using 
a cardinality argument of Hirsch and Hodkinson \cite{HH}. 
Call this class ${\sf K}_{\alpha}$.

\begin{theorem} The class of completely representable algebras in ${\sf K}_{\alpha}$ is not elementary
\end{theorem}
Let $\C\in {\sf K}_{\alpha}$ such that $\C\models d_{01}<1$.  Such algebras exist, for example one can take $\C$ to be 
$\wp(^{\alpha}2).$ Assume that $f: \C\to \wp(^{\alpha}X)$ is a   complete representation. 
Since $\C\models d_{01}<1$, there is $s\in h(-d_{01})$ so that if $x=s_0$ and $y=s_1$, we have
$x\neq y$.
For any $S\subseteq \alpha$ such that $0\in S$, set $a_S$ to be the sequence with 
$ith$ coordinate is $x$, if $i\in S$ and $y$ if $i\in \alpha\sim S$. 
By complete representability every $a_S$ is in $h(1)$ and so in 
$h(\mu)$ for some unique atom $\mu$. 

Let $S, S'\subseteq \alpha$ be distinct and assume that each contains $0$.   Then there exists 
$i<\alpha$ such that $i\in S$, and $i\notin S'$. So $a_S\in h(d_{01})$ and 
$a_S'\in h (-d_{01}).$ Therefore atoms corresponding to different $a_S$'s are distinct. 
Hence the number of atoms is equal to the number of subsets of $\alpha$ that contain $0$, so it is at least $^{|\alpha|}2$. 
Now using the downward Lowenheim Skolem Tarski theorem, take an elementary substructure $\B$ of $\C$ with $|\B|\leq |\alpha|.$
This is possible since the scope of cylindrifiers and the support of substitutions are $< \kappa$.
Then in $\B$ we have $\B\models d_{01}<1$. But $\B$ has at most $|\alpha|$ atoms, and so $\B$ cannot be completely representable.

This does not hold for Ferenczi's relativized algebras, precisely because they are relativized. 
The above argument depended essentially on the cardinality of the square $^{\alpha}2$.

\subsection{Possible extensions}

Finally, we mention that Ferenczi's ideas can be transferred to the countable paradigm, by using countable transformation systems.
In more detail given a countable ordinal $\alpha$ one defines a strongly rich semigroup $G$ as in \cite{AU}, with $G\subseteq {}^{\alpha}\alpha$.
The signature is now like polyadic equality algebras except that substitutions are restricted to $G$, and cylindrifiers are finite. 
Such semigroups are adequate to define dilations. 
Postulating Ferenczi's  axioms for this new signature, we get all the results obtained in this paper, using {\it relativized} semantics.

An important addition in this new context is an omitting types theorem since now we can apply the Baire Category theorem.
Indeed in such a context one can easily prove an exact analogue of the Orey-Henkin omitting types, omitting even $< covK$ many non principal types,
where the latter, as mentioned earlier is the least cardinal such the  Baire category theorem 
for compact second countable Hausdoff spaces fail, and it is also the largest for which Martin's axiom for countable Boolean  
algebras holds, hence the the best estimate for the number of non principal types that can be omitted in countable theories.
We think that this could be a reasonable solution to the so called finitizability problem 
in algebraic logic, which has been open for ages for the equality case. 
The solution for logics without equality is provided by Sain.

\end{document}